\documentclass[11pt,twoside]{article}

\usepackage{mathrsfs,amsfonts,amsmath,amssymb}
\usepackage{latexsym}
\newtheorem{satz}{Theorem}

\newtheorem{theorem}[satz]{Theorem}
\newtheorem{lemma}[satz]{Lemma}

\newtheorem{corollary}[satz]{Corollary}
\newtheorem{remark}[satz]{Remark}

\def\Z{\mathbb {Z}}
\def\F{\mathbb {F}}

\def\a{\alpha}
\def\S{{ S_{\alpha}}}
\def\C{\mathbb{C}}

\def\d{\delta}
\def\o{\omega}
\def\({\big (}
\def\){\big )}

\def\G{\Gamma}

\def\le{\leqslant}
\def\ge{\geqslant}
\def\_phi{\varphi}

\def\FF{\widehat}

\def\ov{\overline}

\def\S{{\mathcal S}}
\def\R{{\mathbf R}}
\def\Vol{\mathsf{Vol}}

\author{Shkredov I.D.}
\title{ On some problems of Euclidean Ramsey theory
\footnote{
This work was supported by grant
Russian Scientific Foundation RSF 14--11--00433.}
}
\date{}
\begin{document}
\maketitle

\begin{center}
 Annotation.
\end{center}

{\it \small
    In the paper
    we prove, in particular, that for any measurable coloring of the euclidian plane into two colours
    there is a
    monochromatic triangle with some restrictions on the sides.
    Also we consider similar problems in finite fields settings.
}
\\


\section{Introduction}
\label{sec:introduction}

Let $\Pi = \R^2$ be the ordinary euclidian plane.
Any partition of $\Pi$ onto $k$ disjoint sets $C_1, \dots, C_k$ is called {\it $k$--coloring} of $\Pi$
and the sets $C_1, \dots, C_k$ are called {\it colors.}
A well--known unsolved
question
of Euclidian Ramsey Theory (see \cite{EGMRSS}, \cite{Colouring_book}) asks us about the existence of a monochromatic
(that is belonging to
the same
color) non--equilateral triangle (that is just any three points from $\Pi$) in any two--coloring of
the plane.
The
problem
seems to be difficult and only partial results are known, see \cite{Colouring_book}.
In particular,
the question remains open
even for the case of a degenerate triangle, having all three  points lying on a line.
A parallel and even more famous problem in the area is to find the {\it chromatic number} of the plane $\chi (\R^2)$,
that is the smallest number of colors sufficient for coloring the plane in such
a way that no two points of the same color are unit distance apart.
It is well--known that $4 \le \chi (\R^2) \le 7$.
In his beautiful paper \cite{F_chromatic} Falconer proved that if all colors are {\it measurable} sets then the correspondent {\it measurable chromatic number} of the plane is at least five.
Our paper is devoted to a measurable analog of the considered two--coloring problem.
The main
result is
the following, see Theorem \ref{t:euclid} and Theorem \ref{t:euclid'}
from section \ref{sec:euclid}.

\begin{theorem}
    Let $ABC$ be a nondegenarate triangle such that $|AB|/|AC| = \omega$.
    Suppose that
$$
    \min_{t\ge 0} (J_0 (t) + J_0 (\omega t)) \ge -0.5972406 \,,
$$
    where $J_0$ is the zeroth Bessel function.
    Then any measurable coloring of $\R^2$ into two colors contains a monochromatic triangle.\\
    Further, if
$$
        \min_{t\ge 0} (J_0 (t) + J_0 (\kappa t) + J_0 ((1+\kappa)t)) > - 1 \,.
$$
    Then for any measurable coloring of the plane $\Pi$ into two colors
    there is a monochromatic collinear triple $\{ x, y, z\}$ such that $y \in [x,z]$ and
    $\| z - y\| / \| y-x\| = \kappa$.
\label{t:main_intr}
\end{theorem}

The proof uses simple Fourier analysis in spirit of paper \cite{OFV}
and hugely relies on the fact that we
have deal
just
with two colors.
Also we consider a model situation of the plane over the prime finite field $\F_p \times \F_p$ and prove
(a slightly stronger) analog of Theorem \ref{t:main_intr}, see section \ref{sec:finite}.
The proof develops the method from \cite{BHIPR}, \cite{IK}, \cite{Wolf_AP4}.

The author is grateful to R. Prasolov and J. Wolf for useful discussions.

\section{Finite fields case}
\label{sec:finite}

Let $p$ be a prime number, and $\F_p$ be the prime field.
Let also $\Pi = \F_p \times \F_p$ be the prime plane.
If $x\in \Pi$ then we write $x=(x_1,x_2)$.
For any $j\neq 0$ define a {\it sphere} in $\Pi$, that is the set
$$
    \S_j = \{ x\in \Pi ~:~ \| x \| := x_1^2 + x_2^2 = j \} \,.
$$
For any function $f: \Pi \to \C$ denote its Fourier transform as
$$
    \FF{f} (r) := \sum_{x\in \Pi} f(x) e^{-\frac{2 \pi i (x_1 r_1 + x_2 r_2 )}{p}}
        =
            \sum_{x\in \Pi} f(x) e^{-\frac{2 \pi i \langle x, r\rangle}{p}}
                =
                    \sum_{x\in \Pi} f(x) e(-\langle x, r\rangle) \,.
$$
The inverse formula takes place
\begin{equation}\label{f:inverse}
    f(x) = p^{-2} \sum_{r \in \Pi} \FF{f}(r) e(\langle r, x \rangle) \,.
\end{equation}
For any
two functions $f,g : \Pi \to \C$ the Parseval identity holds
\begin{equation}\label{f:Parseval}
    \sum_{x \in \Pi} f(x) \ov{g(x)} = p^{-2} \sum_{r \in \Pi} \FF{f} (r) \ov{\FF{g}(r)} \,.
\end{equation}
Further, put
$$
    (f*g) (x) := \sum_{y\in \Pi} f(y) g(x-y) \quad \mbox{ and } \quad
        (f\circ g) (x) := \sum_{y\in \Pi} f(y) g(y+x) \,.
$$
Then
\begin{equation}\label{f:F_svertka}
    \FF{f*g} = \FF{f} \FF{g} \quad \mbox{ and } \quad \FF{f \circ g} = \FF{f^c} \FF{g} = \ov{\FF{\ov{f}}} \FF{g} \,,
\end{equation}
where for a function $f:\Pi \to \mathbb{C}$ we put $f^c (x):= f(-x)$.
 Clearly,  $(f*g) (x) = (g*f) (x)$ and $(f\circ g)(x) = (g \circ f) (-x)$, $x\in \Pi$.
If $A\subseteq \Pi$ is a set then denote by $A(x)$ its characteristic function.

\bigskip

Using Gauss and Kloosteman sums one can prove the following rather standard lemma, see e.g.
\cite{IK}.
Exact formula for the cardinalities of the spheres in $\Pi$ can be obtained as well.

\bigskip

\begin{lemma}
    We have
\begin{equation}\label{f:sheres_Fourier_1}
    |\S_j| = p + 2 \theta \sqrt{p} \,,
\end{equation}
    where $|\theta| \le 1$,
    and for all $r\neq 0$ one has
\begin{equation}\label{f:sheres_Fourier_2}
    |\FF{\S}_j (r)| \le 2 \sqrt{p} \,.
\end{equation}
    Moreover, for any invertible  $\mathbf{g} :\Pi \to \Pi$ and all $r\neq 0$
    the following holds
\begin{equation}\label{f:sheres_Fourier_3}
    |\FF{\mathbf{g} (\S_j)} (r)| \le 2 \sqrt{p} \,.
\end{equation}
\label{l:sheres_Fourier}
\end{lemma}
\begin{proof}
    We will prove just (\ref{f:sheres_Fourier_3}), the proof of (\ref{f:sheres_Fourier_1}), (\ref{f:sheres_Fourier_2})
    is similar and is contained in \cite{IK}, Lemma 2.
    Put $\S = \S_j$.
    We have
$$
    \FF{\mathbf{g} (\S)} (r) = \sum_x \S (\mathbf{g}^{-1} x) e(-\langle x, r\rangle)
        =
            \sum_x \S (x) e(-\langle \mathbf{g} (x), r\rangle)
                =
                    p^{-1} \sum_{k\in \F_p} \sum_x e(k (\| x\| - j) -\langle \mathbf{g} (x), r\rangle)
$$
$$
    = p^{-1} \sum_{k \neq 0} e(-kj) \sum_x e(k \| x\| -\langle \mathbf{g} (x), r\rangle)
        =
            p^{-1} \sum_{k \neq 0} e(-kj) \sum_x e(k \| x\| - a x_1 - b x_2) \,,
$$
where $a,b\in \F_p$ are some constants depending of $r$.
Completing the square and using the well--known formula
$$
    G(\a) := \sum_z e(\a z^2) = \left( \frac{\a}{p} \right) G (1) \,,
$$
where $G (1) =\sum_z e(z^2)$ is the Gauss sum and $\left( \frac{\a}{p} \right)$ is the Legendre symbol, we obtain
$$
    \FF{\mathbf{g} (\S)} (r) = \frac{G^2 (1)}{p} \sum_{k \neq 0} e(-kj - c k^{-1}) \,,
$$
where $c$ is some constant.
Now applying $|G(1)| = \sqrt{p}$ and the estimate for the Kloosterman sums \cite{Weil}, we get
$$
    |\sum_{k \neq 0} e(-kj - c k^{-1})| \le 2 \sqrt{p} \,.
$$
This completes the proof.
$\hfill\Box$
\end{proof}

\bigskip

For any set $A\subseteq \Pi$ denote by $f_A (x)$ the {\it balanced function} of the set $A$, that is
$f_A (x) = A(x) - |A|/|\Pi|$.
Clearly, $\sum_x f_A (x) = 0$.
By $I : \Pi \to \Pi$ denote the identity map.

\begin{theorem}
    Let $p$ be a sufficiently large prime number.
    Suppose that $\mathbf{g}$ is an invertible
    affine
    transformation of $\Pi$ such that $\mathbf{g} - I$ is also invertible.
    Then for any two--coloring of the plane $\Pi$ and any $a\neq 0$ there is a monochromatic
    triple $\{ x, y, z\}$ such that $y=x+s$, $s\in \S_a$ and $z=x+\mathbf{g} (s)$.
\label{t:F_p_2}
\end{theorem}
\begin{proof}
    Let $\S = \S_a$ and $A,B$ be the colors of our coloring.
    We are interested into the quantity
$$
    \sigma (A) := \sum_{x} \sum_{s\in \S} A(x) A(x+s) A(x+\mathbf{g} (s)) \,,
$$
    and similar for the color $B$.
    Let us rewrite the quantity $\sigma (A)$ in terms of the balanced function of $A$.
    Put $\d_A = |A|/|\Pi|$, $\d_B = |B| / |\Pi|$.
    Then because of the balanced function has zero mean, we get
$$
    \sigma (A) := \sum_{x} \sum_{s\in \S} (\d_A + f_A) (x) (\d_A + f_A) (x+s) (\d_A + f_A) (x+\mathbf{g} (s))
        =
$$
$$
        =
            \d^3_A |\S| p^2
                +
                    \d_A ( \sum_{x} \sum_{s\in \S} f_A (x) f_A (x+s)
                        + \sum_{x} \sum_{s\in \S} f_A (x) f_A (x+\mathbf{g} (s))
                            +
$$
$$
                            +
                                \sum_{x} \sum_{s\in \S} f_A (x+\mathbf{g} (s)) f_A (x+s) )
                                    +
                                        \sum_{x} \sum_{s\in \S} f_A (x) f_A (x+s) f_A (x+\mathbf{g} (s))
                                        =
$$
\begin{equation}\label{tmp:30.06.2015_1}
    =
        \d^3_A |\S| p^2 + \d_A (\sigma_1 + \sigma'_1 + \sigma''_1) + \sigma_2 \,.
\end{equation}
Let us estimate $\sigma_1$.
By formulas (\ref{f:Parseval}), (\ref{f:F_svertka}), we obtain
$$
    \sigma_1 = \sum_s \S(s) (f_A \circ f_A) (s)
     = p^{-2} \sum_r \FF{\S} (r) |\FF{f}_A (r)|^2
$$
and thus, using the Parseval identity once more time
as well as
Lemma \ref{l:sheres_Fourier}, we get
$$
    \sigma_1 = p^{-2} \sum_{r\neq 0} \FF{\S} (r) |\FF{f}_A (r)|^2
        \le
            2 \sqrt{p} \cdot p^{-2} \sum_{r} |\FF{f}_A (r)|^2
            \le
            2 \sqrt{p} |A| \,.
$$
So, it is negligible comparing the main term in (\ref{tmp:30.06.2015_1}).
Now by the invertibility of $\mathbf{g}$, we have
$$
    \sigma'_1 = \sum_s \S(s) (f_A \circ f_A) (\mathbf{g} (s)) =
    \sum_s \S(\mathbf{g}^{-1} (s) ) (f_A \circ f_A) (s)
$$
and we can apply the arguments above because of
one can
use
bound
(\ref{f:sheres_Fourier_3})  of Lemma \ref{l:sheres_Fourier}
instead of (\ref{f:sheres_Fourier_2}).
Finally
$$
    \sigma''_1 = \sum_s \S(s) (f_A \circ f_A) (\mathbf{g} (s) - s)
        =
            \sum_s \S((\mathbf{g} - I)^{-1} (s) ) (f_A \circ f_A) (s)
$$
and by the invertibility of $\mathbf{g} - I$
and in view of
Lemma \ref{l:sheres_Fourier}
we can estimate $\sigma''_1$ similarly as $\sigma'_1$.

It remains to
calculate the quantity $\sigma_2$.
Using
the inverse formula  (\ref{f:inverse})
it is easy to see that
$$
    \sigma_2 = \sigma_2 (A) = p^{-4} \sum_{u,v} \FF{f}_A (-u-v) \FF{f}_A (u) \FF{f}_A (v)
    \cdot \left( \sum_s \S (s) e(\langle s, u \rangle + \langle \mathbf{g} (s), v \rangle) \right) \,.
$$
Because of $A(x)+B(x) = 1$ we have $\FF{f}_A (r) = - \FF{f}_B (r)$ for all $r\in \Pi$.
It follows that $\sigma_2 (A) + \sigma_2 (B) = 0$.
Another way to see the fact is to check the identity $f_A (x) + f_B (x) = 0$.
Whence, using Lemma \ref{l:sheres_Fourier} again, we obtain
$$
    \sigma (A) + \sigma (B) \ge |\S| p^2 (\d^3_A + \d^3_B) - 6 \sqrt{p} |A| - 6 \sqrt{p} |B|
        \ge
            \frac{|\S| p^2}{4} - 6 p^2 \sqrt{p}
                \ge
                    \frac{p^3}{4} - 6.5 p^2 \sqrt{p} > 0 \,,
$$
provided by $p>1000$, say.
This completes the proof.
$\hfill\Box$
\end{proof}

\bigskip

Because of
two
distinct
points of $\Pi$ can be transformed to another
pair
by a composition of an orthogonal transformation and a dilation (see e.g. \cite{BHIPR}, \cite{HI_simplexes}) then we obtain two immediate
consequences of the theorem above.

\begin{corollary}
    Let $p$ be a sufficiently large prime number, $p\equiv -1 \pmod 4$,
    and three points $A,B,C \in \Pi$ form a non--equilateral triangle.
    Then for any two--coloring of the plane $\Pi$ there is a
    monochromatic triangle
    congruent to $\triangle ABC$.
\label{c:finite1}
\end{corollary}
\begin{proof}
First of all note that
any triangle has a pair of sides such that the quotient of its "lengths"\, is a quadratic residue.
Let $\| A-B\| = a$, $\|A-C\| = b$, and  $a/b$ be a quadratic residue.
Then there is an affine transformation $\mathbf{g}$ (which is a composition of an orthogonal map of $\Pi$ and a dilation, see \cite{HI_simplexes}) such that
$\mathbf{g} (A) = A$, $\mathbf{g} (B) = C$.
It is easy to see that both maps $\mathbf{g}$ and $\mathbf{g}-I$ are invertible.
Indeed, if $\mathbf{g}$ is not invertible then there is $x \neq 0$ such that $\| x \| =0$.
But in view of the assumption $p\equiv -1 \pmod 4$, we derive  $x=0$ with a contradiction
(note that in the case of $p\equiv 1 \pmod 4$ there is $i \in \F_p$ such that $i^2 \equiv -1 \pmod p$
and hence there are $x\neq 0$ with $\| x \| = 0$).
Finally, if $\mathbf{g}-I$ is not invertible then $\triangle ABC$ is equilateral 
and for some $x$ one has $\mathbf{g} x = x$ and
hence $\mathbf{g}$ is a mirror symmetry
(in the case there are some restrictions on the length of the side $a$ of $\triangle ABC$ as 
$a$ is nonresidual and $a+1$ is residual but we miss them).
This completes the proof.
$\hfill\Box$
\end{proof}

\begin{corollary}
    Let $p$ be a sufficiently large prime number.
    Then for any two--coloring of the plane $\Pi$ and any $a,b \neq 0$
    such that
    $a/b$ is a quadratic residue
    there is a monochromatic
    collinear triple $\{ x, y, z\}$
    with
    $\| y-x\| = a$, $\| z-y\| = b$.
\label{c:finite2}
\end{corollary}


{\bf Problem.}
Is it possible to find larger monochromatic configurations from $\Pi$ in the spirit of papers \cite{CCS}, \cite{Wolf_AP4}?


\section{Euclidian plane}
\label{sec:euclid}

In the section we consider the case of the  usual euclidian plane and try to obtain an analog of Theorem \ref{t:F_p_2}.
The proof
follows
the arguments from \cite{OFV} as well as the approach (and the notation) from the previous section.

Put
$\Pi = \R^2$ and
let
$A\subseteq \Pi$ be a measurable set.
By the {\it upper density} of $A$ define
\begin{equation}\label{f:upper_density}
    \ov{\d}_A := \limsup_{T\to +\infty} \frac{\Vol (A\cap [-T,T]^2)}{(2T)^2} \,.
\end{equation}

A measurable, complex valued function $f: \Pi \to \C$ is called {\it periodic} if there is a basis $b_1,b_2 \in \Pi$ such that for all $\a_1,\a_2 \in \Z$ one has $f(x+\a_1 b_1 + \a_2 b_2) = f(x)$.
The set $L=\{  \a_1 b_1 + \a_2 b_2 ~:~ \a_1,\a_2 \in \Z \}$ is called the {\it period lattice} of $f$ and
$L^* = \{ u \in \Pi ~:~ \langle u,x \rangle \in \Z,\, \forall x\in L \}$ is called the {\it dual lattice} of $L$.
Here $\langle \cdot, \cdot \rangle$ is the usual scalar product onto $\Pi$.
If the characteristic function of a measurable set $A$ is periodic then it is easy to see that $\limsup$ in (\ref{f:upper_density}) can be replaced
by a
simple limit.

We have a scalar product on the space of periodic functions
$$
    \langle f, g \rangle = \lim_{T \to +\infty} \frac{1}{(2T)^2} \int_{[-T,T]^2} f(x) \ov{g(x)}\, dx \,.
$$
The Fourier transform of a (periodic) function $f$ is given by the formula
$\FF{f} (u) = \langle f(x), e^{i \langle u, x\rangle} \rangle$.
It is easy to check that the support of Fourier transform of a periodic function $f$
belongs to $2\pi L^*$.
In particular, the support  is a discrete set.

The Bessel function of the first kind $J_\nu (z)$ is the series (see e.g. \cite{Special_functions})
\begin{equation}\label{f:Bessel_series}
    J_\nu (z) = \sum_{k=0}^{+\infty} \frac{(-1)^k (z/2)^{v+2k}}{k! \G (\nu+k+1)} \,.
\end{equation}
It is well--known that
\begin{equation}\label{f:Bessel_Fourier}
    J_0 (\| u\|) = \frac{1}{2\pi} \langle \S_1 (x), e^{iux} \rangle \,,
\end{equation}
where $J_0$ is the zeroth Bessel function and
$$\S_a = \{ x\in \Pi ~:~ \| x \| := \sqrt{x_1^2+x_2^2} = a\}$$
is a circle of radios $a$.

\begin{theorem}
    Let $a>0$ and $\kappa > 0$ be real numbers.
    Suppose that for all $t\ge 0$ one has
\begin{equation}\label{c:euclid}
    J_0 (t) + J_0 (\kappa t) + J_0 ((1+\kappa)t) > - 1 \,.
\end{equation}
    Then for any measurable coloring of the plane $\Pi$ into two colors
    there is a monochromatic collinear triple $\{ x, y, z\}$ such that $y \in [x,z]$ and
    $\| y-x\| = a$, $\| z - y\| = \kappa a$.
\label{t:euclid}
\end{theorem}
\begin{proof}
We follow the arguments of the proof of Theorem \ref{t:F_p_2}.
Let $\S = \S_a$ be the circle of radios $a$ and $A,B$ be the colors of upper densities $\ov{\d}_A$, $\ov{\d}_B$.
We suppose that $A$ and $B$ do not contain collinear triples $\{ x, y, z\}$ such that
$y \in [x,z]$ and $\| y-x\| = a$, $\| z - y\| = \kappa a$.

One can
assume
that $A(x)$ and $B(x)$ are periodic functions.
Indeed, choose $T$ is large enough that $[-T+(1+\kappa)a,T-(1+\kappa)a]^2 / (2T)^2$
is sufficiently close to $1$ and such that
$\Vol (A\cap [-T,T]) / (2T)^2$ is sufficiently close to $\ov{\d}_A$.
After that construct a periodic tiling of $\R^2$ with copies of $A\cap [-T+(1+\kappa)a,T-(1+\kappa)a]^2$
and $B\cap [-T+(1+\kappa)a,T-(1+\kappa)a]^2$,
translating the copies by the points of lattice $2T \Z$.
Denote the obtained new colors as $A_*$ and $B_*$.
Clearly, $\d_{A_*}$ can be chosen is close to $\ov{\d}_{A}$ and that $\d_{A_*} + \d_{B_*} = 1$.
Note also that $A_*$, $B_*$ do not contain collinear triples $\{ x, y, z\}$ such that
$y \in [x,z]$ and $\| y-x\| = a$, $\| z - y\| = \kappa a$.

As
in the proof of Theorem \ref{t:F_p_2}
consider the quantity $\sigma (h_1,h_2,h_3)$,  which is trilinear by three arguments $h_1,h_2,h_3$, namely,
$$
    \sigma (A_*,A_*,A_*) = \sigma (A_*) :=
$$
$$
        \lim_{T \to +\infty} \frac{1}{(2T)^2}
            \int \int_{s\in \S} (\d_{A_*} + f_{A_*}) (x) (\d_{A_*} + f_{A_*}) (x+s) (\d_{A_*} + f_{A_*}) (x-\kappa s) \,dxds \,,
$$
where  again
$f_{A_*} (x) = A_* (x) - \d_{A_*}$ is the balanced function of $A_*$.
We have
\begin{equation}\label{tmp:03.07.2015_-1}
    \sigma (A_*) :=
            2\pi a \d^3_{A_*}
                +
                    \d_A ( \sigma (f_{A_*}, f_{A_*}, 1) + \sigma (f_{A_*}, 1, f_{A_*}) + \sigma (1, f_{A_*}, f_{A_*}) )
                +
                                        \sigma (f_{A_*}) \,.
\end{equation}
As in the proof of Theorem \ref{t:F_p_2} the following holds
$\sigma (f_{A_*}) + \sigma (f_{B_*}) = 0$
and thus we need to bound the remain three quantities in (\ref{tmp:03.07.2015_-1}).
Clearly,
$$
    \sigma (f_{A_*}, f_{A_*}, 1)
        =
            \langle f_{A_*} \circ f_{A_*}, \S \rangle \,.
$$
Using the Fourier transform, we get
\begin{equation}\label{tmp:03.07.2015_1}
    \sigma (f_{A_*}, f_{A_*}, 1) = \sum_{u\in \R^2} |\FF{f}_{A_*} (u)|^2 \FF{\S} (u) \,.
\end{equation}
As we noted before the sum in (\ref{tmp:03.07.2015_1}) is actually taking over a discrete set.
Putting
$$
    \a (t) := \sum_{u \in \R^2 ~:~ \| u\| = t} |\FF{f}_{A_*} (u)|^2 \ge 0 \,,
$$
we obtain by (\ref{f:Bessel_Fourier})
\begin{equation*}\label{tmp:03.07.2015_2}
    \sigma (f_{A_*}, f_{A_*}, 1) = 2\pi a \sum_{t \ge 0} J_0 (at) \a (t) \,.
\end{equation*}
Here we have used the formula
$$
    \FF{\S}_b (u) = b \FF{\S}_1 (b u) = 2\pi b J_0 (\| b u\|) \,,
$$
where $b>0$ is an arbitrary.
Similarly,
$$
    \sigma (f_{A_*}, 1, f_{A_*}) = \langle (f_{A_*} \circ f_{A_*}) (\kappa s), \S (s) \rangle
        =
            2\pi a \sum_{t\ge 0} J_0 (\kappa at) \a(t) \,,
$$
and
$$
    \sigma (1, f_{A_*}, f_{A_*}) = \langle (f_{A_*} \circ f_{A_*}) ((1+\kappa) s), \S (s) \rangle
        =
            2\pi a \sum_{t\ge 0} J_0 ((1+\kappa)at) \a(t) \,.
$$
Thus
$$
    \sigma (f_{A_*}, f_{A_*}, 1) + \sigma (f_{A_*}, 1, f_{A_*}) + \sigma (1, f_{A_*}, f_{A_*})
        \ge
            2\pi a \sum_{t\ge 0} \a(t) (J_0 (at) + J_0 (\kappa at) + J_0 ((1+\kappa)at) \,.
$$
By $J$ define the quantity $J=\min_{t\ge 0} (J_0 (at) + J_0 (\kappa at) + J_0 ((1+\kappa)at)$.
Applying  the Parseval identity and the observation $\FF{A}_* (0) = \d_{A_*}$, we have
\begin{equation}\label{tmp:07.07.2015_Par}
    \sum_{t\ge 0} \a (t) = \sum_{u} |\FF{A}_* (u)|^2 - \d^2_{A_*} = \d_{A_*} - \d^2_{A_*} \,.
\end{equation}
Returning to (\ref{tmp:03.07.2015_-1}) and combining it with the last formula, we obtain
$$
    (2\pi a)^{-1} (\sigma (A_*) + \sigma (B_*)) \ge
        \d^3_{A_*} + \d^3_{B_*} + J (\d^2_{A_*} - \d^3_{A_*}) + J (\d^2_{B_*} - \d^3_{B_*})
            =
                (\d^3_{A_*} + \d^3_{B_*}) (1-J) + (\d^2_{A_*} + \d^2_{B_*}) J \,.
$$
Because of $\d_{A_*}+ \d_{B_*}=1$ the optimization gives us
$$
    (2\pi a)^{-1} (\sigma (A_*) + \sigma (B_*)) \ge \frac{J+1}{4} > 0 \,.
$$
Here we have used condition (\ref{c:euclid}).
This completes the proof.
$\hfill\Box$
\end{proof}

\begin{corollary}
    Let $a>0$ be a real number.
    Then for any measurable coloring of the plane $\Pi$ into two colors
    there is a monochromatic collinear triple $\{ x, y, z\}$ such that $y \in [x,z]$ and
    $\| y-x\| = \| z - y\| = a$.
\end{corollary}
\begin{proof}
By Theorem \ref{t:euclid}, we need to estimate
$\min_{t\ge 0} (2J_0 (at) + J_0 (2at)) = \min_{t\ge 0} (2J_0 (t) + J_0 (2t))$.
Using Maple, say, one can calculate $\min_{t\in [0,50]} (2J_0 (t) + J_0 (2t)) \ge -0.74$.
For $t>50$, applying a crude upper bound $|J_\nu (t)| \le |t|^{-1/3}$, $\nu \ge 0$
(see e.g. \cite{Bessel_upp}),
we
insure
that the minimum is strictly greater than $-1$
for all $t\ge 0$.
This concludes the proof.
$\hfill\Box$
\end{proof}

\bigskip 

Below we will deal with affine transformations $\mathbf{g}$ of the form 
$\mathbf{g} = D_\o \circ R$,  
where $R$ be a rotation  and $D_\o$ be a dilation by some $\o>0$.
Let us note a simple lemma about such $\mathbf{g}$.

\begin{lemma}
    Let $\mathbf{g} = D_\o \circ R$, where $D_\o$ be a dilation by $\o$ and $R$ be a rotation by $\_phi$. 
    Then $\mathbf{g}-I$ 
    has the same form 
    $D_{\o'} \circ R'$, where
    $\o' = \sqrt{\o^2 - 2\o \cos \_phi +1}$ and $R'$ is another rotation. 
\label{l:g-I}
\end{lemma}
\begin{proof}
To obtain the result we need to solve the system of equations 
$\o \cos \_phi - 1 =\o' \cos \_phi'$, $\o' \sin \_phi' = \o \sin \_phi$ in variables $\o', \_phi'$. 
Taking a square and a summation give us $$(\o')^2 =  \o^2 - 2\o \cos \_phi +1 \ge 0$$ 
and thus $\sin \_phi' = \o / \o' \cdot \sin \_phi$, $\cos \_phi' = (\o \cos \_phi - 1)/\o'$. 
On can check that the modules of $\o / \o' \cdot \sin \_phi$
as well as $(\o \cos \_phi - 1)/\o'$ do not exceed $1$ and hence $\_phi'$ exists.  
This completes the proof.
$\hfill\Box$
\end{proof}

\bigskip

Similarly to Theorem \ref{t:euclid}
as well as
Theorem \ref{t:F_p_2}, one can obtain the following general result, which is however 
not so wide as 
Theorem \ref{t:F_p_2}.

\begin{theorem}
    Let $a>0$ and $\o > 0$ be real numbers.
    Let also $\mathbf{g} = D_\o \circ R$ be an affine  transformation of $\Pi$,
    where $R$ be a rotation  and $D_\o$ be a dilation by $\o$.
    Suppose that for all $t\ge 0$ one has
\begin{equation}\label{c:euclid'}
    J_0 (t) + J_0 (\o t) + J_0 > - 1 \,,
\end{equation}
    where $J_0 = \min_{t\ge 0} J_0 (t) = -0.4027593957...$.
    Then for any measurable coloring of the plane $\Pi$ into two colors
    there is a monochromatic collinear triple $\{ x, y, z\}$ such that $y=x+s$, $s\in \S_a$
    and $z = x + \mathbf{g} (s)$.
    More precisely, if $R$ is a rotation by $\_phi$ then condition (\ref{c:euclid'}) can be replaced by 
\begin{equation}\label{c:euclid''}
    J_0 (t) + J_0 (t \o) + J_0 (t \sqrt{\o^2 - 2\o \cos \_phi +1}) > - 1 \,.
\end{equation}
\label{t:euclid'}
\end{theorem}
\begin{proof}
We use the notation and the arguments of the proof of Theorem \ref{t:euclid}.
Then
$$
    \sigma (A_*,A_*,A_*) = \sigma (A_*) :=
$$
$$
        \lim_{T \to +\infty} \frac{1}{(2T)^2}
            \int \int_{s\in \S} (\d_{A_*} + f_{A_*}) (x) (\d_{A_*} + f_{A_*}) (x+s)
                (\d_{A_*} + f_{A_*}) (x+\mathbf{g} (s)) \,dxds \,,
$$
$$
    =
            2\pi a \d^3_{A_*}
                +
                    \d_A ( \sigma (f_{A_*}, f_{A_*}, 1) + \sigma (f_{A_*}, 1, f_{A_*}) + \sigma (1, f_{A_*}, f_{A_*}) )
                +
                                        \sigma (f_{A_*}) \,.
$$
Again, we need to estimate
$
    \sigma (f_{A_*}, f_{A_*}, 1), \sigma (f_{A_*}, 1, f_{A_*}), \sigma (1, f_{A_*}, f_{A_*})
$.
The first quantity is the same as in the proof of Theorem \ref{t:euclid}.
The second one equals
\begin{equation}\label{tmp:22.07.2015_1}
    \sigma (f_{A_*}, 1, f_{A_*}) = \langle (f_{A_*} \circ f_{A_*}) (\mathbf{g} (s)), \S (s) \rangle
        =
            \langle (f_{A_*} \circ f_{A_*}) (s), \S (\mathbf{g}^{-1} (s)) \rangle \cdot \det (\mathbf{g})^{-1} \,.
\end{equation}
As we know for any $b>0$ one has
\begin{equation}\label{tmp:22.07.2015_2}
    \FF{\S}_b (u) = b \FF{\S}_1 (b u) = 2\pi b J_0 (\| b u\|) \,.
\end{equation}
Hence
\begin{equation}\label{tmp:22.07.2015_3}
    \langle \S (\mathbf{g}^{-1} (s)), e^{i \langle u, s \rangle} \rangle
    = 2\pi a J_0 (\o a\| u\|) \cdot \det (\mathbf{g}) \,.
\end{equation}
In terms of quantities $\a (t)$ it follows that
$$
    \sigma (f_{A_*}, 1, f_{A_*}) = 2\pi a \sum_{t\ge 0} \a(t) J_0 (\o a t) \,.
$$
Finally, in the estimation of the third term $\sigma (1, f_{A_*}, f_{A_*})$ the quantity $(\mathbf{g}-I)^{-1}$ appears.
Hence (see the proof of Theorem \ref{t:F_p_2}), we get
\begin{equation}\label{tmp:04.07.2015_1}
    \sigma (f_{A_*}, 1, f_{A_*}) = \sum_{u\in \R^2} |\FF{f}_{A_*} (u)|^2
        \langle \S ((\mathbf{g}-I)^{-1} (s)), e^{i \langle u, s \rangle} \rangle \,.
\end{equation}
Unfortunately, if $u$ runs over a circle then $(\mathbf{g}-I)^{-1} (u)$
do not belong to
a circle in the case of general transformation $\mathbf{g}$ 
(but it is so in the case of Theorem \ref{t:euclid} when $\{x,y,z\}$ are collinear).
Nevertheless we estimate (\ref{tmp:04.07.2015_1}) with help of (\ref{tmp:07.07.2015_Par}) crudely as
$$
    \sigma (f_{A_*}, 1, f_{A_*}) \ge 2\pi a J_0 \cdot (\d_{A_*} - \d^2_{A_*}) \,.
$$
Combining all bounds, we obtain
$$
    (2\pi a)^{-1} (\sigma (A_*) + \sigma (B_*)) \ge \frac{J+J_0+1}{4} > 0 \,,
$$
where $J=\min_{t\ge 0} (J_0 (at) + J_0 (\o at)) = \min_{t\ge 0} (J_0 (t) + J_0 (\o t))$.
Thus, we have proved (\ref{c:euclid'}) and it remains to obtain (\ref{c:euclid''}). 
In the case apply Lemma \ref{l:g-I}, combining with formula (\ref{tmp:04.07.2015_1}) 
and calculations in (\ref{tmp:22.07.2015_1})---(\ref{tmp:22.07.2015_3}).
This completes  the proof.
$\hfill\Box$
\end{proof}

\begin{remark}
    It is well--known that there is a measurable two--coloring of the plane having no monochromatic
    equilateral triangle of an arbitrary side $a>0$, see \cite{EGMRSS}.
    If we try to apply Theorem \ref{t:euclid'} in the case then a Maple calculation gives us
    $$
        \min_{t\ge 0} (2J_0 (t)) + J_0 = 3 J_0 =
        - 1.208278187...
    $$
    It is rather close to the required $-1$.
\end{remark}

There is a series of results, see e.g. \cite{Colouring_book} where the existence of monochromatic triangle with some restrictions on the lengths  of the sides and the angles was obtained for any (non--necessary measurable) coloring.
For example, in \cite{EGMRSS} the authors proved that any monochromatic
triangle with the smallest side $1$ and the angles in the ratio
$1:2:3$, more generally, in the ratio $n:(n+1):(2n+1)$, $1:2n:(2n+1)$ and so on can be found.
Concluding the section we note
that in our Theorem \ref{t:euclid'}
one
does  not need to know any angles but just the ratio of the lengths of an arbitrary  two sides of the triangle.
For example, one can show that for $\o=2$ the minimum in (\ref{c:euclid'})
is greater that $-0.86$ and hence any monochromatic triangle with the ratio of the sides $1:2$ appears.

\bigskip

\noindent{I.D.~Shkredov\\
Steklov Mathematical Institute,\\
ul. Gubkina, 8, Moscow, Russia, 119991}
\\
and
\\
IITP RAS,  \\
Bolshoy Karetny per. 19, Moscow, Russia, 127994\\
{\tt ilya.shkredov@gmail.com}

\end{document}